\documentclass{amsart}
\usepackage{mathrsfs,amssymb,mathrsfs, multirow,setspace}
\usepackage[a4paper, total={7in, 9in}]{geometry}
\usepackage{enumerate}
\theoremstyle{plain}
\usepackage{graphicx}
\usepackage{mathtools}

\newtheorem{theorem}{Theorem}[section]
\newtheorem{lemma}[theorem]{Lemma}
\newtheorem{proposition}[theorem]{Proposition}

\theoremstyle{definition}

\newtheorem{definition}[theorem]{Definition}

\theoremstyle{remark}
\newtheorem{remark}[theorem]{Remark}

\input xy
\xyoption{all}

\begin{document}
	\title[Upper ideal relation graphs associated to rings]{Upper ideal relation graphs associated to rings}
	\author[Barkha Baloda, Jitender Kumar]{Barkha Baloda, $\text{Jitender Kumar}^{^*}$}
	\address{Department of Mathematics, Birla Institute of Technology and Science Pilani, Pilani, India}
	\email{barkha0026@gmail.com,jitenderarora09@gmail.com}

\begin{abstract}
Let $R$ be a ring with unity. The upper ideal relation graph $\Gamma_U(R)$ of the ring $R$ is a simple undirected graph whose vertex set is the set of all non-unit elements of $R$ and two distinct vertices $x, y$ are adjacent if and only if there exists a non-unit element $z \in R$ such that the ideals $(x)$ and $(y)$ contained in the ideal $(z)$. In this article, we classify all the non-local finite commutative rings whose upper ideal relation graphs are split graphs, threshold graphs and cographs, respectively. In order to study topological properties of $\Gamma_U(R)$, we determine all the non-local finite commutative rings $R$ whose upper ideal relation graph has genus at most $2$. Further, we precisely characterize all the non-local finite commutative rings for which the crosscap of $\Gamma_U(R)$ is either $1$ or $2$.
\end{abstract}

\subjclass[2020]{05C25}
\keywords{Non-local commutative rings, ideals, genus, crosscap, graph classes.\\ *  Corresponding author}
\maketitle
\section{Introduction}
The study of algebraic structures through graphs has emerged as a fascinating research discipline over the past three decades. It has not only provided exciting results but also opened up a whole new domain yet to be explored.  
Topological graph theory is principally related to the embedding of a graph into a surface without edge crossing. Graphs on surfaces form a natural connection between discrete and continuous mathematics. They enable us to understand both graphs and surfaces better. Its applications lie in printing electronic circuits, where the aim is to embed a circuit that is the graph on a circuit board (the surface) without two connections crossing each other and thereby resulting in a short circuit. If a graph $\Gamma$ can be drawn on a plane such that its edges intersects only at their end points, then the graph $\Gamma$ is called planar. All the graphs are not planar but they can be embedded on topological surfaces like $g$-hole torus, projective plane and Klein bottle etc. A graph is said to be \emph{embeddable} on a topological surface if it can be drawn on a surface without edge crossing. The \emph{genus}  of a graph $\Gamma$ is the minimum integer $g$ such that $\Gamma$ can be embedded in an orientable surface with $g$ handles. The graphs having genus 0 and genus $1$ are called \emph{planar} graphs and \emph{toroidal} graphs, respectively. Let $\mathbb{N}_k$ be the non-orientable surface formed by the connected sum of $k$ projective planes. The \emph{crosscap} of a graph $\Gamma$ is the minimum non-negative integer $k$ such that $\Gamma$ can be embedded in $\mathbb{N}_k$. A graph is called \emph{projective planar}, if it's crosscap is $1$. In topological graph theory determining the genus and crosscap of a graph is a basic but a very complicated problem, it is indeed NP-complete.

 The topological invariants like planarity, genus and crosscap are extensively studied for graphs associated to algebraic structures.  All the finite dimensional vector spaces have been determined in \cite{kalaimurugan2021genus,tamizh2020genus}, whose non-zero component graphs and non-zero component union graphs, respectively, are of genus at most two. The genus or crosscap of various graphs associated with groups including commuting graphs \cite{afkhami2015planar}, reduced power graphs \cite{anitha2020characterization}, nilpotent graphs \cite{das2015genus} and non-cyclic graphs \cite{ma2020finite}, have been investigated in the literature and all the finite groups have been ascertained with corresponding graphs having genus either 0 or 1 or 2.  The association of graphs to rings was introduced by Beck \cite{beck1988coloring} and he was mainly interested in the coloring of a  graph associated to a commutative ring. Then Anderson and Livingston \cite{anderson1999zero} studied a subgraph of the graph introduced by Beck and named as the zero divisor graph.  Further, various other graphs associated to rings, namely: Cayley graphs, total graphs, zero divisor graphs, annihilating-ideal graphs, generalized co-maximal graphs and co-maximal graphs, have been introduced and studied from different perspectives. The problem of investigating the genus and crosscap of graphs associated to rings has been considered by many researchers. For instance, Belshoff and Chapman \cite{belshoff2007planar} classified all finite local rings whose zero divisor graph is planar. Wang \cite{wang2006zero} determined all the rings of the form $\mathbb{Z}_{p_1^{m_1}} \times \mathbb{Z}_{p_2^{m_2}} \times \cdots \times \mathbb{Z}_{p_k^{m_k}}$ or $\frac{\mathbb{Z}_n[x]}{(x^m)}$ whose zero divisor graphs have genus at most one. All the non-local commutative rings whose zero-divisor graphs having genus either $1$ or $2$, have been classified in \cite{asir2020classification, wickham2008classification}. 
 Jesili \emph{et al.} \cite{jesili2023genus} characterized all the commutative non-local rings whose reduced cozero-divisor graph has genus at most one. Biswas \emph{et al.} \cite{biswas2022subgraph} provided all the finite commutative rings whose generalised co-maximal graph has genus at most two.
All commutative rings whose total graph has genus at most two have been characterized in \cite{maimani2012rings,chelvam2013genus}. All commutative rings whose total graph has crosscap either $1$ or $2$ have been obtained in \cite{asir2019classification, khashyarmanesh2013projective}.


Since the ideals play a vital role in the development of ring structures, it is interesting to study graphs associated with the ideals of a ring. In this connection, the graphs associated with ideals including inclusion ideal graphs \cite{akbari2015inclusion}, intersection ideal graphs \cite{chakrabarty2009intersection},  co-maximal ideal graphs \cite{ye2012co} etc., of rings have been studied in the literature. Asir \emph{et al.} \cite{asir2018classification} determined all isomorphic classes of commutative rings whose ideal-based total graph has genus $1$ or $2$. Planar and toroidal graphs that are intersection ideal graphs of Artinian commutative rings are classified in \cite{pucanovic2014toroidality}. Pucanovi\'{c} \emph{et al.} \cite{pucanovic2014genus} characterized all the graph classes of genus $2$ that are intersection graphs of ideals of some commutative rings. The intersection ideal graphs with crosscap at most two of all the Artinian commutative rings have been investigated by Ramanathan \cite{ramanathan2021projective}.
To reveal the relationships between the ideals and the elements of the ring $R$, Ma \emph{et al.} \cite{ma2016automorphism} introduced the ideal-relation graph of $R$, denoted by $\overrightarrow{\Gamma_{i}}(R)$, is a directed graph whose vertex set is $R$ and there is an edge from a vertex $x$ to a distinct vertex $y$ if and only if $(x) \subset (y)$. Analogously, the undirected ideal-relation graph of the ring $R$ is the simple graph whose vertex set is $R$ and two distinct vertices $x, y$ are adjacent if and only if either $(x) \subset (y)$ or $(y) \subset (x)$, that is the principal ideals $(x)$ and $(y)$ are comparable in the poset of principal ideals of $R$. Thus, it is natural to consider a new graph associated to the ring $R$ such that its vertices $x$ and $y$ are adjacent if and only if  $(x)$ and $(y)$ have an upper bound in the poset of the principal ideals of $R$. Consequently, the \emph{upper ideal relation} graph of the ring $R$ has been introduced by Baloda \emph{et al.} \cite{baloda2023study}. The upper ideal relation graph $\Gamma_U(R)$ of the ring $R$ is a simple undirected graph whose vertex set is the set of all non-unit elements of $R$ and two distinct vertices $x, y$ are adjacent if and only if there exists a non-unit element $z \in R$ such that the ideals $(x)$ and $(y)$ contained in the ideal $(z)$. Various graph-theoretic properties including completeness, planarity and the minimum degree of $\Gamma_U(R)$ have been investigated in \cite{baloda2023study}.  Moreover, the vertex connectivity, automorphism group, Laplacian and the normalized Laplacian spectrum of the upper ideal relation graph $\Gamma_U(\mathbb{Z}_n)$ of the ring  of integers modulo $n$ have been determined in \cite{baloda2023study}.

The aim of this paper is to investigate embeddings of upper ideal relation graphs $\Gamma_U(R)$ on certain surfaces. We classify all the non-local commutative rings $R$ such that $\Gamma_U(R)$ is of genus (or crosscap) at most two. Moreover, we study some important graph classes including split graphs, threshold graphs and cographs, of $\Gamma_U(R)$. These graph classes can be defined either structurally or in terms of forbidden induced subgraphs. This paper is arranged as follows. In Section 2, we recall some definitions and necessary results. In Section $3$, we classify all the non-local commutative rings whose upper ideal relation graphs are split graphs, threshold graphs and cographs. Section 4 classifies all the non-local commutative rings $R$ for which $\Gamma_U(R)$ has genus at most $2$. Also, we determine precisely all the non-local commutative rings for which $\Gamma_U(R)$ has crosscap either $1$ or $2$. The manuscript is concluded with some future research problems in Section $5$.


\section{preliminaries}

In this section, first we recall the graph theoretic notions from  \cite{westgraph, white1985graphs}. 
A \emph{graph} $\Gamma$ is a pair  $\Gamma = (V, E)$, where $V(\Gamma)$ and $E(\Gamma)$ are the set of vertices and edges of $\Gamma$, respectively. Two distinct vertices $u_1, u_2$ are $\mathit{adjacent}$, denoted by $u_1 \sim u_2$, if there is an edge between $u_1$ and $u_2$. Otherwise, we write it as $u_1 \nsim u_2$. Let $\Gamma$ be a graph. A \emph{subgraph}  $\Gamma'$ of $\Gamma$ is the graph such that $V(\Gamma') \subseteq V(\Gamma)$ and $E(\Gamma') \subseteq E(\Gamma)$. If $X \subseteq V(\Gamma)$ then the subgraph of $\Gamma$ induced by $X$, denoted by $\Gamma(X)$, is the graph with vertex set $X$ and two vertices of $\Gamma(X)$ are adjacent if and only if they are adjacent in $\Gamma$. A graph $\Gamma$ is said to be $complete$ if every two distinct vertices are adjacent. The complete graph on $n$ vertices is denoted by $K_n$. A graph $\Gamma$ is said to be a \emph{complete bipartite graph} if the vertex $V(\Gamma)$ can be partitioned into two disjoint union of nonempty sets $A$ and $B$, such that two distinct vertices are adjacent if and only if they belong to different sets. Moreover, if the cardinality $|A|$ of $A$ is $m$ and $|B| = n$ then we denote it by $K_{m,n}$. A \emph{path} in a graph is a sequence of distinct vertices with the property that each vertex in the sequence is adjacent to the next vertex of it. A path graph on $n$ vertices is denoted by $P_n$. A \emph{cycle} is a path that begins and ends on the same vertex. If a cycle is of length $n$ then we write it as $C_n$. A subset $X$ of $V(\Gamma)$ is said to be \emph{independent} if no two vertices of $X$ are adjacent. A graph $\Gamma$ is a \emph{split} graph if the vertex set is the disjoint union of  two sets $A$ and $B$, where $A$ induces a complete subgraph and $B$ is an independent set.
\begin{lemma}\label{splitgraph}\cite{MR0505860}
A graph $\Gamma$ is a split graph if and only if it does not have an induced subgraph isomorphic to one of the three forbidden graphs, $C_4, C_5$ or $2K_2$.
\end{lemma}

A graph $\Gamma$ is said to be a \emph{cograph} if it has no induced subgraph isomorphic to $P_4$. A \emph{threshold} graph is the graph which does not contain an induced subgraph isomorphic to $P_4, C_4$ or $2K_2$. Every threshold  graph is a cograph as well as split graph. A \emph{cactus} graph is a connected graph where any two simple cycles have at most one vertex in common. A connected graph is said to be \emph{unicyclic} if it contains exactly one cycle. A graph $\Gamma$ is \emph{outerplanar} if it can be embedded in the plane such that all vertices lie on the outer face.  A graph $\Gamma$ is \emph{planar} if it can be drawn on a plane without edge crossing. It is well known that every outerplanar graph is a planar graph. Now we have the following known results related to outerplanar and planar graphs. 
\begin{theorem}\cite{westgraph}\label{outerplanar criteria}
A graph $\Gamma$ is outerplanar if and only if it  does not contain a subdivision of $K_4$ or $K_{2,3}$.
\end{theorem}
\begin{theorem}\cite{westgraph}\label{planar criteria}
A graph $\Gamma$ is planar if and only if it does not contain a subdivision of $K_5$ or $K_{3,3}$.
\end{theorem}
Recall that a chord in a graph $\Gamma$ is an edge joining two non adjacent vertices in a cycle of $\Gamma$. A cycle $C$ is said to be \emph{primitive} if it has no chord. A graph $\Gamma$ satisfies  the primitive cycle property (PCP) if any two primitive cycles intersect in at most one edge. The \emph{free rank} of graph $\Gamma$, denoted by $frank(\Gamma)$, is the number of primitive cycles in $\Gamma$. The cycle rank $rank(\Gamma)$ of $\Gamma$ is the number $|E(\Gamma)|-|V(\Gamma)| + \mathcal{C}$, where $\mathcal{C}$ is the number of connected components of $\Gamma$.


A compact connected topological space such that each point has a neighbourhood homeomorphic to an open disc in $\mathbb{R}^2$ is called a surface. An embedding of a graph $\Gamma$ on a surface $\mathbb{S}$ is $2-$cell embedding if each component of $\mathbb{S}- \Gamma$ is homeomorphic to an open disc in $\mathbb{R}^2$. A $2-$cell embedding is said to be \emph{triangular} if all the faces have boundary consisting of exactly three edges. Let $\mathbb{S}_{\mathbb{g}}$ denote the orientable surface with $\mathbb{g}$ handles, where $\mathbb{g}$ is a non-negative integer. The \emph{genus} of a graph $\Gamma$, denoted by $\mathbb{g}(\Gamma)$, is the minimum integer $\mathbb{g}$ such that $\Gamma$ can be embedded in $\mathbb{S}_{\mathbb{g}}$ that is the graph $\Gamma$ can be drawn into a surface $\mathbb{S}_{\mathbb{g}}$ without edge crossing. The following results are useful in the sequel.

\begin{proposition}{\cite[Ringel and Youngs]{white1985graphs}}
\label{genus}
 Let $n \geq 3$ be a positive integer. Then $\mathbb{g}(K_n) =$ $\left\lceil \frac{(n-3)(n-4)}{12} \right \rceil $.
\end{proposition}
 \begin{lemma}{\cite[Theorem 5.14]{white1985graphs}}
 \label{eulerformulagenus}
 Let $\Gamma$ be a connected graph, with a 2-cell embedding in $\mathbb{S}_{\mathbb{g}}$. Then $v - e + f = 2 - 2 \mathbb{g}$, where $v, e$ and $f$ are the number of vertices, edges and faces embedded in $\mathbb{S}_{\mathbb{g}}$, respectively and $\mathbb{g}$ is the genus of surface of a graph embedded. 
\end{lemma}
\begin{lemma}\cite{white2001graphs}\label{genusofblocks}
The genus of a connected graph $\Gamma$ is the sum of the genera of its blocks.
\end{lemma}

Let $\mathbb{N}_{k}$ denotes the non-orientable surface formed by the connected sum of $k$ projective planes, that is, $\mathbb{N}_{k}$ is a non-orientable surface with $k$ crosscap. The \emph{crosscap} of a graph $\Gamma$, denoted by $cr(\Gamma)$, is the minimum non-negative integer $k$ such that $\Gamma$ can be embedded in $\mathbb{N}_{k}$. For instance, a graph $\Gamma$ is planar if $cr(\Gamma) = 0$ and the  $\Gamma$ is projective if $cr(\Gamma) = 1$. The following results are useful to obtain the crosscap of a graph.
\begin{proposition}\label{crosscap}{\cite[Ringel and Youngs]{mohar2001graphs}}
 Let $n$ be a positive integer. Then
 \begin{center}
$ cr(K_n) =
 \begin{cases} 
\left\lceil \frac{(n-3)(n-4)}{6} \right\rceil,&  \textit{if}~~~~~ n\geq 3 \\
    3, & \textit{if}~~~~~ n =7.
   \end{cases}$
   \end{center}
\end{proposition}
\begin{lemma}\label{eulerformula}{\cite[Lemma 3.1.4]{mohar2001graphs}}
Let $\phi : \Gamma \rightarrow \mathbb{N}_{k}$ be a $2$-cell embedding of a connected graph $\Gamma$ to the non-orientable surface $\mathbb{N}_{k}$. Then $v - e + f = 2 - k$, where $v, e$ and $f$ are the number of vertices, edges and faces of $\phi(\Gamma)$ respectively, and $k$ is the crosscap of $\mathbb{N}_{k}$.
\end{lemma}

\begin{definition}\cite{white2001graphs}
 A graph $\Gamma$ is orientably simple if $\mu(\Gamma) \neq 2-cr(\Gamma)$, where $\mu(\Gamma) = \rm{max} \{2- 2\mathbb{g}(\Gamma), 2-cr(\Gamma)\}$.
 
    
\end{definition}
\begin{lemma}\cite{white2001graphs}\label{crosscapofblocks}
Let $\Gamma$ be a graph with blocks $\Gamma_1, \Gamma_2, \cdots, \Gamma_k$. Then 
\begin{center}
$cr(\Gamma) =$
$\begin{cases} 
1-k+ \sum \limits_{i=1}^{k} cr(\Gamma_i),  & ~~\textit{if}~~ \Gamma~~ \textit{is orientably simple}\\
2k - \sum \limits_{i=1}^{k} \mu (\Gamma_i), & ~~\textit{otherwise.}
\end{cases}$
\end{center}
\end{lemma}
We shall use the following remark explicitly without referring to it.
\begin{remark}\label{triangularface}
For a simple graph $\Gamma$, every face has atleast three boundary edges and every edge is a boundary of two faces, that is, $2e \geq 3f$. Moreover, the equality holds if and only if $\Gamma$ has a triangular embedding. 
\end{remark}

A ring $R$ is called \emph{local} if it has a unique maximal ideal $\mathcal{M}$ and we abbreviate this by $(R, \mathcal{M})$. In addition, for $x \in R$, $(x)$ denotes the ideal generated by $x$. An ideal of a ring $R$ is said to be a \emph{maximal principal ideal} if it is a maximal among all the principal ideals of $R$. The set of zero-divisors and the set of units of the ring $R$ are denoted by $Z(R)$ and $U(R)$, respectively. The set of all nonzero elements of $R$ is denoted by $R^*$. Also, the field with $q$ elements is denoted by $\mathbb{F}_q$. For other basic definitions of ring theory, we refer the reader to \cite{atiyah1994introduction}.  
Let $R$ be a non-local finite commutative ring. By the structural theorem, $R$ is uniquely (up to isomorphism) a finite direct product of local rings $R_i$ that is $R \cong R_1 \times R_2 \times \cdots \times R_n$, where $n \geq 2$. 
The following remark is useful for later use.
\begin{remark}\label{localringcardinality}
Let $R$ be a finite local ring and $p$ be a prime number. Then $|R| = p^{\alpha}$, for some positive integer $\alpha$.
\end{remark}
The \emph{union} of two graphs $\Gamma_1$ and $\Gamma_2$, denoted by $\Gamma_1 \cup \Gamma_2$, is the graph with $V(\Gamma_1 \cup \Gamma_2) = V(\Gamma_1) \cup V(\Gamma_2)$ and $E(\Gamma_1 \cup \Gamma_2) = E(\Gamma_1) \cup E(\Gamma_2)$. The \emph{join} of $\Gamma_1$ and  $\Gamma_2$, denoted by $\Gamma_1 \vee \Gamma_2$, is the graph obtained from the union of $\Gamma_1$ and $\Gamma_2$ by adding new edges from each vertex of $\Gamma_1$ to every vertex of $\Gamma_2$. The following remark is easy to observe.
\begin{remark}\label{joinandunionoftwo graphs}
For the fields $F_1$ and $F_2$, if $R \cong F_1 \times F_2$, then we have
\begin{center}
$\Gamma_U(F_1 \times F_2) \cong K_1 \vee (K_{|F_1|-1} \bigcup K_{|F_2|-1})$.
\end{center}
\end{remark}
\section{Graph classes of $\Gamma_U(R)$}
In this section, we classify all the non-local finite commutative rings whose upper ideal relation graphs are split graphs, threshold graphs, cographs, cactus graphs and unicyclic graphs. We begin with a classification of non-local finite commutative rings $R$ such that $\Gamma_U(R)$ is a split graph in the following theorem.  

\begin{theorem}\label{splitgraphtheorem}
Let $R$ be a non-local finite commutative ring.  Then $\Gamma_U(R)$ is a split graph if and only if $R$ is isomorphic to  $\mathbb{Z}_2 \times \mathbb{Z}_2 \times \mathbb{Z}_2$ or $\mathbb{Z}_2 \times \mathbb{F}_q$.
\end{theorem}
\begin{proof}
First suppose that $\Gamma_U(R)$ is a split graph. Since $R$ is a non-local commutative ring, we have $R \cong R_1 \times R_2 \times \cdots \times R_n$, where each $R_i$ is a local ring and $n \geq 2$. If $n \geq 4$, then $\Gamma_U(R)$ has a subgraph induced by  $u_1 = (1, 1, 0, 1, \cdots, 1)$, $u_2 = (0, 1, 0, 1, \cdots, 1)$ and $v_1 = (1, 0, 1, 1, \cdots, 1)$, $v_2 = (1, 0, 1, 0, 1, \cdots, 1)$ which is isomorphic to $2K_2$; a contradiction. Thus, either $R \cong R_1 \times R_2 \times R_3$ or $R \cong R_1 \times R_2$. Let $a_1, a_2 \in U(R_1)$, $b_1, b_2 \in U(R_2)$ and $c_1, c_2 \in U(R_3)$. First suppose that $R \cong R_1 \times R_2 \times R_3$. If $|R_i| \geq 3$ for every $i \in \{1, 2, 3\}$, then notice that $u_1 = (a_1, 0, c_1)$, $u_2 = (a_2, 0, c_2)$ and $v_1 = (0, b_1, c_1)$, $v_2 = (0, b_2, c_2)$ induces a subgraph of $\Gamma_U(R)$ which is isomorphic to $2K_2$. Without loss of generality, assume that $|R_1| = 2$ and $|R_2| = 3= |R_3|$. Then for the set $X = \{(0, b_1, c_1), (0, b_2, c_2)  (1, b_1, 0), (1, b_2, 0)\}$ we have $\Gamma_U(X) \cong 2K_2$; a contradiction. Therefore, both $R_2$ and $R_3$ can not have cardinality three. We may now suppose that $|R_1| = 2 = |R_2|$ and $|R_3| = 3$. The set $X' = \{(0, 1, c_1), (0, 1, c_2), (1, 0, c_1), (1, 0, c_2)\}$ induces $2K_2$ as a subgraph of $\Gamma_U(R)$ which is not possible. Consequently, $R \cong \mathbb{Z}_2 \times \mathbb{Z}_2 \times \mathbb{Z}_2$.

Now, suppose that $R \cong R_1 \times R_2$. If $|R_1| \geq 3$ and $|R_2| \geq 3$ then the set of vertices $\{(a_1, 0), (a_2, 0), (0, b_1), (0, b_2)\}$ induces a subgraph isomorphic to $2K_2$; a contradiction. Without loss of generality, assume that $|R_1| = 2$ and $|R_2| \geq 3$. If $R_2$ is not a field then there exists $z \in Z(R_2^*)$. Consequently, the subgraph induced by $X'' = \{(0, b_1), (0, b_2), (1, 0), (1, z)\}$ is isomorphic to $2K_2$, which is not possible. Thus, $R \cong \mathbb{Z}_2 \times \mathbb{F}_q$.

Conversely,  suppose that $R \cong \mathbb{Z}_2 \times \mathbb{Z}_2 \times \mathbb{Z}_2$. Note that $V(\Gamma_U(R)) = \mathcal{V}_1 \cup \mathcal{V}_2$, where $\mathcal{V}_1 = \{(1, 1, 0), (1, 0, 1), (0, 1, 1)\}$ is an independent set and $\mathcal{V}_2 = \{(0, 0, 0), (1, 0, 0), (0, 1, 0), (0, 0, 1)\}$ forms a complete subgraph of $\Gamma_U(R)$. Moreover, $\mathcal{V}_1 \cap \mathcal{V}_2 = \emptyset$. Thus, by definition, $\Gamma_U(R)$ is a split graph. If $R \cong \mathbb{Z}_2 \times \mathbb{F}_q$ then note that $V(\Gamma_U(R)) = \mathcal{V}_1 \cup \mathcal{V}_2$, where $\mathcal{V}_1 = \{(1, 0)\}$ and $\mathcal{V}_2 = \{(0, b) : b \in \mathbb{F}_q\}$. Consequently, the result holds.
\end{proof}
The following theorem describes all the non-local commutative rings whose upper ideal relation graph is a threshold graph.
\begin{theorem}
Let $R$ be a non-local finite commutative ring. Then $\Gamma_U(R)$ is a threshold graph if and only if $R$  is isomorphic to $\mathbb{Z}_2 \times \mathbb{F}_q$.
\end{theorem}
\begin{proof}
Since $R$ is a non-local finite commutative ring, we have $R \cong R_1 \times R_2 \times \cdots \times R_n$, where each $R_i$ is a local ring and $n \geq 2$. Suppose that $\Gamma_U(R)$ is a threshold graph. Then $\Gamma_U(R)$ is a split graph also. 
 By Theorem \ref{splitgraphtheorem},  either $R \cong \mathbb{Z}_2 \times \mathbb{Z}_2 \times \mathbb{Z}_2$ or $R \cong \mathbb{Z}_2 \times \mathbb{F}_q$. If $R \cong \mathbb{Z}_2 \times \mathbb{Z}_2 \times \mathbb{Z}_2$, then there exists an induced subgraph $\Gamma_U(X)$, where\\ $X = \{(1, 0, 1), (1, 0, 0), (0, 1, 0), (0, 1, 1)\}$, which is isomorphic to $P_4$ which is not possible. Consequently, $R \cong \mathbb{Z}_2 \times \mathbb{F}_q$.
 
Conversely, if $R \cong \mathbb{Z}_2 \times \mathbb{F}_q$, then by Remark \ref{joinandunionoftwo graphs}, we obtain $\Gamma_U(R) \cong K_1 \vee (K_1 \cup K_{|\mathbb{F}_q|-1})$. It follows that $\Gamma_U(R)$ is a threshold graph.
\end{proof}

Now in the following theorem, we characterize all the non-local finite commutative rings whose upper ideal relation graphs are cographs i.e. $P_4$-free.  

\begin{theorem}
Let $R$ be a non-local finite commutative ring such that $R \cong R_1 \times R_2 \times \cdots \times R_n$, $(n \geq 2)$ and each $(R_i, \mathcal{M}_i)$ is a local ring. Then $\Gamma_U(R)$ is a cograph  if and only if  $R \cong R_1 \times R_2$ and $\mathcal{M}_1, \mathcal{M}_2$ are the maximal principal ideals.
\end{theorem}
\begin{proof}
First, suppose that  $\Gamma_U(R)$ is a cograph. If $n \geq 3$, consider the set
\begin{center}
$X = \{(1, 0, 1, \cdots, 1), (1, 0, \cdots, 0), (0, 1, 0, \cdots, 0), (0, 1, \cdots, 1)\}$.
\end{center}
Notice that $\Gamma_U(X) \cong P_4$; a contradiction. Thus,  $R \cong R_1 \times R_2$, where $(R_1, \mathcal{M}_1)$ and $(R_2, \mathcal{M}_2)$ are local rings. Now we show that both the ideals $\mathcal{M}_1$ and $\mathcal{M}_2$ are maximal principal. Without loss of generality, assume that $\mathcal{M}_1$ is not a maximal principal ideal. Then $R_1$ has atleast two maximal principal ideals, namely: $(x_1)$ and $(x_2)$. For instance, for some $y \in R_1$, if $(y)$ is the only maximal principal ideal of $R_1$ then $\mathcal{M}_1 = R_1 \setminus U(R_1) \subseteq (y)$. Consequently, $\mathcal{M}_1 = (y)$; a contradiction. Further, note that $x_1 \nsim x_2$ in $\Gamma_U(R_1)$. Moreover, $(x_1, 1) \sim (x_1, 0) \sim (x_2, 0) \sim (x_2, 1)$ is an induced subgraph which is isomorphic to $P_4$; a contradiction. Thus, both $\mathcal{M}_1$ and $\mathcal{M}_2$ must be the maximal principal ideals of $R_1$ and $R_2$, respectively.

Conversely, suppose that $R \cong R_1 \times R_2$ and $\mathcal{M}_i$ is the principal ideals of $R_i$. To prove $\Gamma_U(R)$ is a cograph, consider the sets
\begin{center}
\hspace{-0.8cm}$V_1 = \{(z_1, z_2) : z_1 \in \mathcal{M}_1, z_2 \in \mathcal{M}_2 \}$;\\
    $V_2 = \{(z_1, u_2) : z_1 \in \mathcal{M}_1, u_2 \in R_2 \setminus {\mathcal{M}_2}\}$;\\
     $V_3 = \{(u_1, z_2) : u_1 \in R_1 \setminus {\mathcal{M}_1}, z_2 \in \mathcal{M}_2 \}$.
\end{center}
Observe that  $V_1, V_2$ and  $V_3$ forms a partition of $V(\Gamma_U(R))$. Since $\mathcal{M}_1$ is the principal ideal and $\mathcal{M}_1 = V(\Gamma_U(R_1))$, we obtain $\Gamma_U(R_1) \cong K_{|\mathcal{M}_1|}$. Similarly, $\Gamma_U(R_2) \cong K_{|\mathcal{M}_2|}$. Consequently,  $\Gamma_U(V_i) \cong K_{|V_i|}$ for each $i \in \{1, 2, 3\}$. Let $x = (z_1, z_2) \in V_1$. If $y = (t_1, t_2) \in V_2$ then note that $(x), (y) \subseteq ((z, t_2))$, where $\mathcal{M}_1 = (z)$. It follows that $x \sim y$ in $\Gamma_U(R_1 \times R_2)$. Similarly, $x \sim y$ for every $y \in V_3$. Consequently, $x \sim y$ for every $y \in V(\Gamma_U(R_1 \times R_2))$. Further, note that for each $x \in V_2$ and $y \in V_3$, we have $x \nsim y$ in $\Gamma_U(R_1 \times R_2)$. Thus,  $\Gamma_U(R) \cong K_{|V_1|} \vee (K_{|V_2|} \bigcup K_{|V_3|})$. Hence, $\Gamma_U(R)$ is a cograph.
\end{proof}

In the next two theorems, we precisely determine all the non-local finite commutative rings whose upper ideal relation graphs are cactus and unicyclic, respectively. 
\begin{theorem}
Let $R$ be a non-local finite commutative ring. Then $\Gamma_U(R)$ is a cactus graph if and only if $R$ is isomorphic to one of the following $3$ rings: $\mathbb{Z}_2 \times \mathbb{Z}_2$, $\mathbb{Z}_2 \times \mathbb{Z}_3$, $\mathbb{Z}_3 \times \mathbb{Z}_3$.
\end{theorem}
\begin{proof}
Let $R$ be a non-local finite commutative ring. Then $R \cong R_1 \times R_2 \times \cdots \times R_n$, where each $R_i$ is a local ring and $n \geq 2$.
First suppose that $\Gamma_U(R)$ is a cactus graph. For $n \geq 3$, the graph $\Gamma_U(R)$ has two cycles (infact, triangles)
\begin{center}
$\mathcal{C}_1: (1, 0, \cdots, 0) \sim (1, 1, 0, \cdots, 0) \sim (0, 1, 0, \cdots, 0) \sim (1, 0, \cdots,  0)$; and \\
$\mathcal{C}_2: (0, 1, 0, \cdots, 0) \sim (0, 0, 1, \cdots, 0) \sim (1, 0, \cdots, 0) \sim (0, 1, 0, \cdots, 0)$,
\end{center}
 which has a common edge $(1, 0, \cdots, 0) \sim (0, 1, 0, \cdots, 0)$; a contradiction. Thus, $R \cong R_1 \times R_2$. Without loss of generality, assume that $|R_1| \geq 4$ with $a_1, a_2, a_3 \in R_1^*$. Then the cycles $\mathcal{C}_1: (a_1, 0) \sim (0, 0) \sim (a_2, 0) \sim (a_1, 0)$ and $\mathcal{C}_2: (0, 0) \sim (a_3, 0) \sim (a_1, 0) \sim (0, 0)$ has a common edge. Therefore, $|R_i| \leq 3$ for every $i \in \{1, 2\}$. Thus, the result holds. Converse follows from Remark \ref{joinandunionoftwo graphs}.

\end{proof}

\begin{theorem}
Let $R$ be a non-local finite commutative ring. Then $\Gamma_U(R)$ is unicyclic  if and only if $R \cong \mathbb{Z}_2 \times \mathbb{Z}_3$.
\end{theorem}
\begin{proof}
Let $R$ be a non-local finite commutative ring. Then $R \cong R_1 \times R_2 \times \cdots \times R_n$, where each $R_i$ is a local ring and $n \geq 2$. First suppose that $\Gamma_U(R)$ is a unicyclic graph. If $n \geq 3$, then $\Gamma_U(R)$ has two cycles $\mathcal{C}_1$ and $\mathcal{C}_2$, where
\begin{center}
$\mathcal{C}_1 : (1, 0, \cdots, 0) \sim (0, 1, 0, \cdots, 0) \sim (0, 0, 1, 0, \cdots, 0) \sim (1, 0, \cdots, 0)$; and \\
$\mathcal{C}_2 : (1, 0, \cdots, 0) \sim (1, 0, 1, \cdots, 0) \sim (0, 0, 1,0, \cdots, 0) \sim (1, 0, \cdots, 0)$.
\end{center}
Therefore, $R \cong R_1 \times R_2$.
Now suppose that $|R_i| \geq 3$ for each $i = 1, 2$  with  $a_1, a_2 \in R_1^*$ and $b_1, b_2 \in R_2^*$. Note that for the sets $X_1 = \{(a_1, 0), (0, 0), (a_2, 0)\}$ and  $ X_2 = \{(0, 0), (0, b_1), (0, b_2)\}$, we get $\Gamma_U(X_i) \cong C_3$, where  $i \in \{1, 2\}$; a contradiction to the fact that $\Gamma_U(R)$ has a unique cycle. We may now suppose that $R \cong R_1 \times R_2$ with $|R_i| \leq 2$ for some $i$. Without loss of generality, assume that $|R_1| =2$. 
If $|R_2| \geq 4$ and $b_1, b_2, b_3 \in R_2^*$, then for the set $Y_1 = \{(0, 0), (0, b_1), (0, b_2)\}$ and $Y_2 = \{(0, 0), (0, b_3), (0, b_2)\}$, we get $\Gamma_U(Y_i) \cong C_3$ for $i \in \{1, 2\}$; a contradiction. Consequently, $|R_1| = 2$ and $|R_2| \leq 3$. If $|R_2| = 2$, then by Remark \ref{joinandunionoftwo graphs}, we obtain $\Gamma_U(R) \cong P_3$. It follows that $\Gamma_U(R)$ does not have a cycle. Thus, $|R_2| = 3$ and so $R \cong \mathbb{Z}_2 \times \mathbb{Z}_3$. Converse is straightforward from Remark \ref{joinandunionoftwo graphs}.
\end{proof}

\section{Embedding of $\Gamma_U(R)$ on surfaces}
In this section, we study the embedding of the upper ideal relation graph on a surface without edge crossing. We begin with the investigation of an embedding of $\Gamma_U(R)$ on a plane.
\subsection{Planarity of $\Gamma_U(R)$}
In this subsection, we classify all the non-local finite commutative rings for which the graph $\Gamma_U(R)$ is outerplanar and planar, respectively. The following lemma is essential to explore the planarity of $\Gamma_U(R)$.

\begin{lemma}\label{notaplanargraph}
Let $R$ be a non-local finite commutative ring such that  $R \cong R_1 \times R_2 \times \cdots \times R_n$ for $n \geq 4$. Then the graph $\Gamma_U(R)$ is not planar.
\end{lemma}
\begin{proof}
Consider the set $X = \{(0,0, \cdots, 0), (1, 0, \cdots, 0), (0, 1, 0, \cdots, 0), (0, 0, 1, 0, \cdots, 0), (0, 0, 0, 1, 0, \cdots, 0)\}$. Note that $\Gamma_U(X) \cong K_5$. Therefore, by Theorem \ref{planar criteria}, $\Gamma_U(R)$ is not a planar graph.
\end{proof}
Now we characterize all the non-local finite commutative rings whose upper ideal relation graph is outerplanar.
\begin{theorem}\label{outerplanar}
Let $R$ be a non-local finite commutative ring. Then $\Gamma_U(R)$ is outerplanar if and only if $R$ is isomorphic to one of the following $3$ rings:
\begin{center}
    $\mathbb{Z}_2 \times \mathbb{Z}_2$, $\mathbb{Z}_2 \times \mathbb{Z}_3$, $\mathbb{Z}_3 \times \mathbb{Z}_3$.
\end{center}
\end{theorem}
\begin{proof}
Let $R$ be a non-local finite commutative ring. Then $R \cong R_1 \times R_2 \times \cdots \times R_n$, where each $R_i$ is a local ring and $n \geq 2$. Let $\Gamma_U(R)$ be an outerplanar graph.  By Lemma \ref{notaplanargraph}, we must have $n \leq 3$. Suppose that $R  \cong R_1 \times R_2 \times R_3$. If $|R_i| = 2$ for every $i \in \{1, 2, 3\}$, then for the set $X = \{(0, 0, 0), (1, 0, 0), (0, 1, 0), (0, 0, 1)\}$ note that $\Gamma_U(X) \cong K_4$, which is not possible. We may now suppose that $R \cong R_1 \times R_2$. Let $|R_i| \geq 4$ for some $i \in \{1, 2\}$. Without loss of generality, assume that $|R_1| = 4$ such that $R_1 = \{0, a_1, a_2, a_3\}$. Then for $X' = \{(0, 0), (a_1, 0), (a_2, 0), (a_3, 0)\}$, we have $\Gamma_U(X') \cong K_4$; again a contradiction. Consequently, $R \cong R_1 \times R_2$ with $|R_i| \leq 3$ for $i \in \{1, 2\}$. Converse holds by Theorem \ref{outerplanar criteria} and Remark \ref{joinandunionoftwo graphs}.
\end{proof}
The following theorem enumerates all the rings $R$ such that the graph $\Gamma_U(R)$ is planar.
\begin{theorem}\label{Planarupperideal}
Let $R$ be a non-local finite commutative ring. Then $\Gamma_U(R)$ is a planar graph if and only if $R$ is isomorphic to one of the following $9$ rings:
\begin{center}
    $\mathbb{Z}_2 \times \mathbb{Z}_2 \times \mathbb{Z}_2,$ $\mathbb{Z}_2 \times \mathbb{Z}_4,$ $\mathbb{Z}_2 \times \frac{\mathbb{Z}_2[x]}{(x^2)}$, $\mathbb{Z}_2 \times \mathbb{F}_4$, $\mathbb{Z}_2 \times \mathbb{Z}_2$, $\mathbb{Z}_2 \times \mathbb{Z}_3$, $\mathbb{Z}_3 \times \mathbb{Z}_3$, $\mathbb{Z}_3 \times \mathbb{F}_4$, $\mathbb{F}_4 \times \mathbb{F}_4$.
\end{center}
\end{theorem}
\begin{proof}
 Suppose that $\Gamma_U(R)$ is a planar graph. In the similar lines of the proof of Theorem \ref{outerplanar}, we have either $R \cong R_1 \times R_2 \times R_3$ or $R \cong R_1 \times R_2$. Let $R \cong R_1 \times R_2 \times R_3$ such that $|R_i| \geq 3$ for some $i \in \{1, 2, 3\}$. Without loss of generality, assume that $|R_1| \geq 3$ with $a_1, a_2 \in R_1^*$. For the set $X = \{(0, 0, 0), (a_1, 0, 0),  (a_2, 0, 0), (0, 1, 0), (0, 0, 1)\}$, we have $\Gamma_U(X) \cong K_5$; a contradiction. Thus, $R \cong \mathbb{Z}_2 \times \mathbb{Z}_2 \times \mathbb{Z}_2$ in this case. We may now suppose that $R \cong R_1 \times R_2$. If $|R_i| \geq 5$ for some $i \in \{1, 2\}$ and $a_1, a_2, a_3, a_4 \in R_i^*$, then for $X' = \{(0, 0), (a_1, 0), (a_2, 0), (a_3, 0), (a_4, 0)\}$ note that $\Gamma_U(X') \cong K_5$ which is not possible. Thus, for $R \cong R_1 \times R_2$ we must have $|R_i| \leq 4$. Further, if $R_2$ is not a field of cardinality four, then by the table given in \cite{nowickitables}, we have either $R_2 \cong \mathbb{Z}_4$ or $R_2 \cong \frac{\mathbb{Z}_2[x]}{(x^2)}$. 
 Then there exists  $z \in Z(R_2^*)$. Then $|R_1| \neq 3, 4$. Otherwise, the set $Y = \{(0, 0), (a_1, 0), (a_2, 0), (0, z), (a_1, z)\}$, where $a_1, a_2 \in R_1^*$, induces a subgraph $\Gamma_U(Y)$ which is isomorphic to $K_5$. Consequently, $R$ is isomorphic to one of the rings: $\mathbb{Z}_2 \times \mathbb{Z}_4,$ $\mathbb{Z}_2 \times \frac{\mathbb{Z}_2[x]}{(x^2)}$, $\mathbb{Z}_2 \times \mathbb{F}_4$, $\mathbb{Z}_3 \times \mathbb{F}_4$, $\mathbb{F}_4 \times \mathbb{F}_4$, $\mathbb{Z}_2 \times \mathbb{Z}_2$, $\mathbb{Z}_2 \times \mathbb{Z}_3$, $\mathbb{Z}_3 \times \mathbb{Z}_3$. Conversely, if $R$ is isomorphic to one of the given rings then by Figures \ref{planardrawingof222}, \ref{planardrawingof24notfield} Theorem \ref{planar criteria} and Remark \ref{joinandunionoftwo graphs}, $\Gamma_U(R)$ is planar.
\begin{figure}[h!]
\centering
\includegraphics[width=0.4 \textwidth]{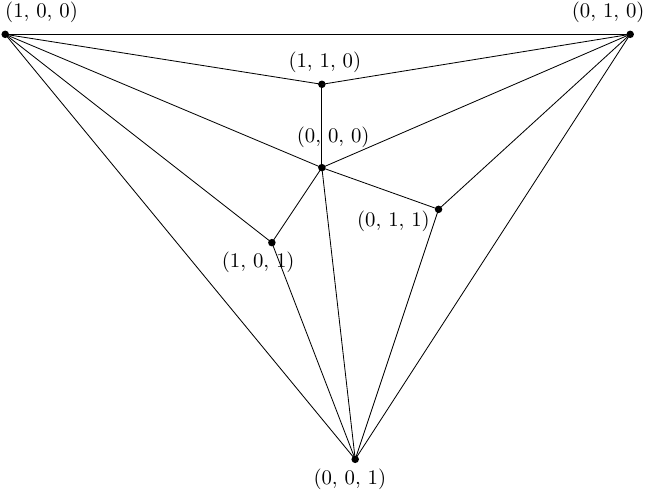}
			\caption{Planar drawing of $\Gamma_U(\mathbb{Z}_2 \times \mathbb{Z}_2 \times \mathbb{Z}_2)$}
			\label{planardrawingof222}
\end{figure}
\begin{figure}[h!]
\centering
\includegraphics[width=0.8 \textwidth]{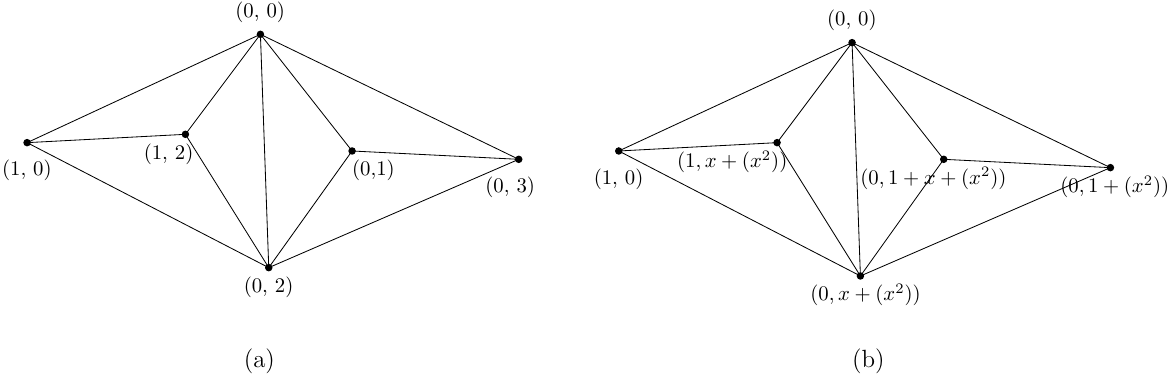}
			\caption{Planar drawing of (a)  $\Gamma_U(\mathbb{Z}_2 \times \mathbb{Z}_4)$ and (b) $\Gamma_U(\mathbb{Z}_2 \times \frac{\mathbb{Z}_2[x]}{(x^2)})$}
			\label{planardrawingof24notfield}
\end{figure}
\end{proof}
A graph $\Gamma$ which satisfies the PCP property is said to be a \emph{ring graph} if $rank(\Gamma) = frank(\Gamma)$ and $\Gamma$ does not contain a subdivision of $K_4$ as a subgraph. In the similar lines of the proof of Theorem \ref{outerplanar} and using Remark \ref{joinandunionoftwo graphs}, we have the following proposition. 
\begin{proposition}
Let $R$ be a non-local finite commutative ring. Then $\Gamma_U(R)$ is a ring graph if and only if $R$ is isomorphic to one of the following $3$ rings:
\begin{center}
    $\mathbb{Z}_2 \times \mathbb{Z}_2$, $\mathbb{Z}_2 \times \mathbb{Z}_3$, $\mathbb{Z}_3 \times \mathbb{Z}_3$.
\end{center}
\end{proposition}
\subsection{Genus of $\Gamma_U(R)$}
In this subsection, we classify all the non-local finite commutative rings such that $\Gamma_U(R)$ has genus $1$ or $2$.

\begin{lemma}\label{genusgreaterthan1}
Let $R$ be a non-local finite commutative ring such that  $R \cong R_1 \times R_2 \times \cdots \times R_n$ for $n \geq 4$. Then $\mathbb{g}(\Gamma_U(R)) > 1$.
\end{lemma}
\begin{proof}
Let $n \geq 4$. Then note that the vertices $x_1 = (0, 0, \cdots, 0),$  $x_2 = (1, 0, \cdots, 0),$ $x_3 = (0, 1, 0, \cdots, 0),$\\ $x_4 = (0, 0, 1, 0, \cdots, 0),$ $x_5 = (0, 0, 0, 1, 0, \cdots, 0),$ $x_6 =(1, 1, 0, 0, \cdots, 0),$ $x_7 = (1, 0, 1, 0, \cdots, 0),$ $x_8 = (1, 0, 0, 1, 0, \cdots, 0)$ induces a subgraph of $\Gamma_U(R)$ which is isomorphic to $K_8$. By Proposition \ref{genus}, we have $\mathbb{g}(\Gamma_U(R)) > 1$.
\end{proof}

All the non-local finite commutative rings with toroidal upper ideal relation graphs have been enumerated in the following theorem.

\begin{theorem}\label{genus1result}
Let $R$ be a non-local finite commutative ring. Then the genus of $\Gamma_U(R)$ is $1$ if and only if $R$ is isomorphic to one of the following $8$ rings:
\begin{center}
    $\mathbb{Z}_2 \times \mathbb{Z}_7,$ $\mathbb{Z}_3 \times \mathbb{Z}_7,$ $\mathbb{F}_4 \times \mathbb{Z}_7,$  
    $\mathbb{Z}_2 \times \mathbb{Z}_5,$ $\mathbb{Z}_3 \times \mathbb{Z}_5,$ $\mathbb{F}_4 \times \mathbb{Z}_5,$  
    $\mathbb{Z}_3 \times \mathbb{Z}_4,$ $\mathbb{Z}_3 \times \frac{\mathbb{Z}_2[x]}{(x^2)}$.
\end{center}
\end{theorem}
\begin{proof}
Let $R$ be a non-local finite commutative ring. Then $R \cong R_1 \times R_2 \times \cdots \times R_n$, where each $R_i$ is a local ring and $n \geq 2$. First suppose that $\mathbb{g}(\Gamma_U(R)) = 1$. By Lemma \ref{genusgreaterthan1}, we get $n \leq 3$. We claim that $n \neq 3$. Let, if possible  $n = 3$ that is $R \cong R_1 \times R_2 \times R_3$. If $|R_i| \geq 3$ for all $i \in \{1,2,3\}$, then the subgraph induced by $X = \{(0, 0, 0), (a_1, 0, 0), (0, b_1, 0), (a_2, 0, 0), (0, b_2, 0), (a_1, b_1, 0), (a_1, b_2, 0), (a_2, b_1, 0)\}$, where $a_1, a_2 \in R_1^*$, $b_1, b_2 \in R_2^*$ and $c_1, c_2 \in R_3^*$, is isomorphic to $K_8$. By Proposition \ref{genus}, we have $\mathbb{g}(\Gamma_U(R)) \neq 1$, a contradiction.  Consequently, $R \cong R_1 \times R_2 \times R_3$ and $|R_i| \leq 2$ for some $i$. Without loss of generality, assume that $|R_1| = 2$. Suppose $|R_i| \geq 3$ for each $i \in \{2, 3\}$. Then by using the set $Y = \{(0, 0, 0), (0, b_1, 0), (0, b_2, 0), (0, 0, c_1), (0, b_1, c_1), (0, b_2, c_1), (0, 0, c_2), (0, b_1, c_2)\}$, we obtain  $\Gamma_U(Y) \cong K_8$; again a contradiction. It implies that $|R_i| = 2$ for some $i \in \{2,3\}$. Without loss of generality, assume that $|R_2| = 2$. If $|R_3| > 3$, then observe that the vertices $x_1 = (0, 0, 0), x_2 = (0, 1, 0), x_3 = (0, 0, c_1), x_4 = (0, 0, c_2), x_5 = (0, 0, c_3), x_6 = (0, 1, c_1), x_7 = (0, 1, c_2), x_8 =(0, 1, c_3)$ induces $K_8$ as a subgraph of $\Gamma_U(R)$, which is not possible. Next, if $|R_3| =3$, then $\Gamma_U(R)$ has $10$ vertices and $31$ edges. Consequently, by Lemma \ref{eulerformulagenus}, we get $f = 21$. Thus, we get $2e < 3f$; a contradiction
 We must have $R \cong R_1 \times R_2 \times R_3$ such that $|R_i| = 2$ for every $i$. But then by Figure \ref{planardrawingof222}, we obtain that $\Gamma_U(R)$ is planar. This completes our claim and so $R \cong R_1 \times R_2$. Now first note that either $|R_1| \geq 8$ or $|R_2| \geq 8$ then there exists an induced subgraph which is isomorphic to $K_8$; a contradiction. It follows that $R \cong R_1 \times R_2$ with $|R_i| \leq 7$ for $i = 1, 2$. Now we classify the ring $R$ such that $\Gamma_U(R)$ has genus $1$ through the following cases.

\noindent\textbf{Case-1:} $|R_2| = 7$. If $|R_1| = 7$ then note that in $\Gamma_U(R)$, $v = 13$, $e = 42$ and $f = 29$. It follows that $2e < 3f$; a contradiction. We may now suppose that $|R_1| = 5$. By  Proposition \ref{genus}, Lemma \ref{genusofblocks} and Remark \ref{joinandunionoftwo graphs}, we get $\mathbb{g}(\Gamma_U(R)) > 1$ which is not possible. Thus, $|R_1| \leq 4$. If $R_1$ is not a field  of cardinality four, then either $R_1 \cong \mathbb{Z}_4$ or $R_1 \cong \frac{\mathbb{Z}_2[x]}{(x^2)}$. Consequently, there exists exactly one zero divisor $z \in Z(R_1^*)$. Then for the set $X' = \{(0,0), (0, 1), \cdots, (0, 6), (z, 0), (z, 1)\}$, we get $\Gamma_U(X') \cong K_9$; a contradiction. Thus, in this case $R$ is isomorphic to one of the three rings: $\mathbb{Z}_2 \times \mathbb{Z}_7,$ $ \mathbb{Z}_3 \times \mathbb{Z}_7,$ $\mathbb{F}_4 \times \mathbb{Z}_7$.

\noindent\textbf{Case-2:}  $|R_2| = 5$. If $|R_1| = 5$,  then by Proposition \ref{genus}, Lemma \ref{genusofblocks} and Remark \ref{joinandunionoftwo graphs}, we get $\mathbb{g}(\Gamma_U(R)) \neq 1$; a contradiction. We may now suppose that $R_1$ is not a field of cardinality four. Then note that the set $\{(0,0), (0,1),(0,2),(0,3),(0,4), (z, 0), (z, 1), (z, 2), (z, 3)\}$, where $z \in Z(R_1^*)$, induces a subgraph which is isomorphic to $K_9$; again a contradiction. Consequently, $R$ is isomorphic to one of the following $3$ rings: $\mathbb{Z}_2 \times \mathbb{Z}_5$, $\mathbb{Z}_3 \times \mathbb{Z}_5$, $\mathbb{F}_4 \times \mathbb{Z}_5$.

\noindent\textbf{Case-3:} $|R_2| = 4$. Suppose that $|R_1| = 4$. If both $R_1$ and $R_2$ are fields, then by Theorem \ref{Planarupperideal}, $\Gamma_U(R_1 \times R_2)$ is planar and so $\mathbb{g}(\Gamma_U(R)) = 0$ which is not possible. Thus, either $R_1$ or $R_2$ is not a field. Without loss of generality, assume that $R_1$ is not a field. By the argument used in Case-1, and by choosing $X'' = \{(0,0), (0, 1), (0, b_1), (0, b_2), (z, 0), (z, 1), (z, b_1), (z, b_2)\}$, where $z \in Z(R_1^*)$, $b_1, b_2 \in R_2^*$, note that $\Gamma_U(X'') \cong K_8$. Consequently, $|R_1| \leq 3$. Let $|R_1| \leq 3$ and $R_2$ be a field. Then by Theorem \ref{Planarupperideal}, $\Gamma_U(R)$ is a planar graph; a contradiction. If $R_2$ is not a field and $|R_1| = 2$ then again by Theorem \ref{Planarupperideal}, $\mathbb{g}(\Gamma_U(R)) = 0$; a contradiction. Thus, $R$ is isomorphic to $\mathbb{Z}_3 \times \mathbb{Z}_4$ or $\mathbb{Z}_3 \times \frac{\mathbb{Z}_2[x]}{(x^2)}$.

\noindent\textbf{Case-4:} $|R_2| \leq 3$. If $|R_1| \in \{2, 3\}$ then by Theorem \ref{Planarupperideal}, $\mathbb{g}(\Gamma_U(R)) = 0$; a contradiction. 
   
Conversely, if $R$ is isomorphic to $\mathbb{Z}_3 \times \mathbb{Z}_4$ or $\mathbb{Z}_3 \times \frac{\mathbb{Z}_2[x]}{(x^2)}$ then by Figure \ref{genusone34}, we have $\mathbb{g}(\Gamma_U(R)) = 1$. If $R$ is isomorphic to one of the remaining $6$ given rings, then by Proposition \ref{genus}, Lemma \ref{genusofblocks} and Remark \ref{joinandunionoftwo graphs}, we get $\mathbb{g}(\Gamma_U(R)) = 1$. 
\begin{figure}[h!]
\centering
\includegraphics[width=0.7 \textwidth]{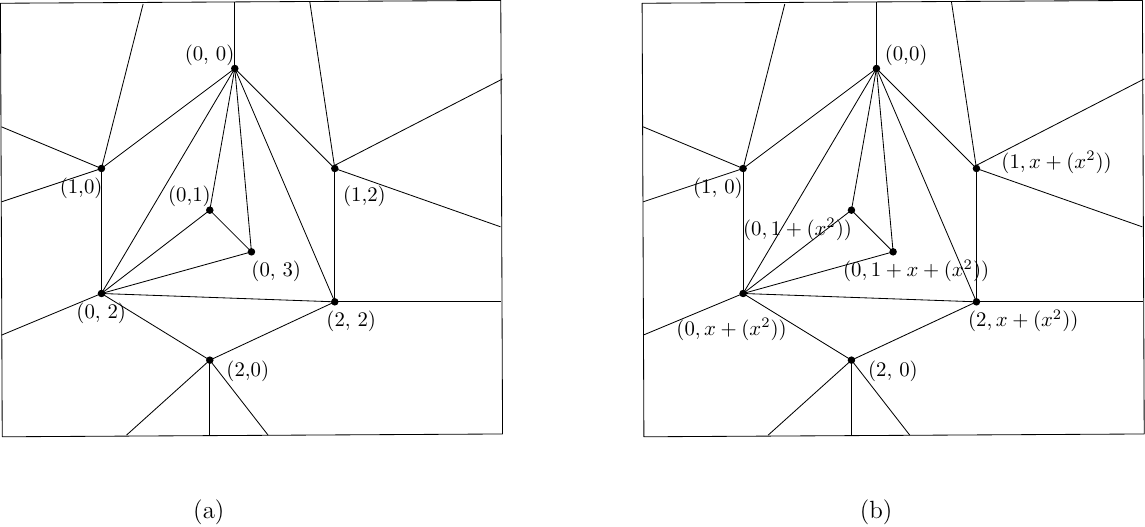}
			\caption{ Embedding of  (a) $\Gamma_U(\mathbb{Z}_3 \times \mathbb{Z}_4)$ and (b) $\Gamma_U(\mathbb{Z}_3 \times \frac{\mathbb{Z}_2[x]}{(x^2)})$ in $\mathbb{S}_1$}
			\label{genusone34}
\end{figure}
\end{proof}
We are now in the position to determine all the non-local finite commutative rings with genus two upper ideal relation graphs.
\newpage
\begin{theorem}\label{genus2result}
Let $R$ be a non-local finite commutative ring. Then $\mathbb{g}(\Gamma_U(R)) = 2$ if and only if $R$ is isomorphic to one of the following $9$ rings:
\begin{center}
$\mathbb{Z}_2 \times \mathbb{Z}_2 \times \mathbb{Z}_3$,
    $\mathbb{Z}_2 \times \mathbb{F}_8,$ $\mathbb{Z}_3 \times \mathbb{F}_8,$ $\mathbb{F}_4 \times \mathbb{F}_8,$  $\mathbb{Z}_7 \times \mathbb{Z}_7,$ $\mathbb{Z}_5 \times \mathbb{Z}_7,$ $\mathbb{Z}_5 \times \mathbb{Z}_5,$  $\mathbb{F}_4 \times \mathbb{Z}_4,$ $\mathbb{F}_4 \times \frac{\mathbb{Z}_2[x]}{(x^2)}$.
\end{center}
\end{theorem}

\begin{proof}
Let $R$ be a non-local finite commutative ring. Then $R \cong R_1 \times R_2 \times \cdots \times R_n$, where each $R_i$ is a local ring and $n \geq 2$. First assume that $\mathbb{g}(\Gamma_U(R)) = 2$. If $n \geq 5$, then note that the vertices $x_1 = (0, 0, \cdots, 0)$,   $x_2 = (1, 0, \cdots, 0),$\\
$x_3 = (0, 1, 0, \cdots, 0),$ $x_4 = (0, 0, 1, 0, \cdots, 0),$ $x_5 = (0, 0, 0, 1, 0, \cdots, 0),$ $x_6 = (0, 0, 0, 0, 1, 0, \cdots, 0),$ $x_7 =(1, 1, 0, 0, \cdots, 0),$ $x_8 = (1, 0, 1, 0, \cdots, 0),$ $x_9 = (1, 0, 0, 1, 0, \cdots, 0)$ induces a subgraph isomorphic to $K_9$; a contradiction.  We may now suppose that $n = 4$. Let $|R_i| > 2$ for some $i \in \{1, 2, 3, 4\}$. Without loss of generality, assume that $|R_1| >2$. Note that the set of vertices $x_1 = (0,0,0,0), x_2 = (a_1, 0, 0,0), x_3 = (0,1,0,0), x_4 = (1,1,0,0), x_5 = (a_1,1,0,0)$, \\ $x_6 =(a_1,1,1,0), x_7 = (0,1,1,0), x_8 = (1,1,1,0), x_9 = (0,0,1,0)$, where $a_1 \in R_1^*$, induces a subgraph isomorphic to $K_9$, a contradiction. Therefore, $|R_i| = 2$ for every $i \in \{1, 2, 3, 4\}$. Thus, in this case, we have $v = 15$, $e = 80$. Further, by Lemma  \ref{eulerformulagenus}, we get $f = 63$. It follows that $2e < 3f$ in $\Gamma_U(R)$, which is not possible. Therefore, $n \leq 3$. Let $n = 3$ so that $R \cong R_1 \times R_2 \times R_3$. If $|R_i| = 2$ for each $i \in \{1, 2, 3\}$, then by Theorem \ref{Planarupperideal}, $\Gamma_U(R)$ is a planar graph. If $|R_i| \geq 3$ for each $i$, then notice that the subgraph of $\Gamma_U(R)$ induced by the set  $\{(0, 0, 0), (a_1, 0, 0), (0, b_1, 0),  (a_2, 0, 0), (0, b_2, 0), (a_1, b_1, 0), (a_1, b_2, 0), (a_2, b_1, 0), (a_2, b_2, 0)\}$, where $a_1, a_2 \in R_1^*$ and $b_1, b_2 \in R_2^*$, is isomorphic to $K_9$; a contradiction. Without loss of generality, assume that $|R_1| = 2$. If $|R_2| \ge 3$ and $|R_3| \ge 3$, then consider the set \[Y = \{(0, 0, 0), (0, 0, c_1), (0, b_1, 0),  (0, 0, c_2), (0, b_2, 0), (0, b_1, c_1), (0, b_2, c_1),(0, b_1, c_2), (0, b_2, c_2)\},\] where $b_1, b_2 \in R_2^*$ and $c_1, c_2 \in R_2^*$. Note that $\Gamma_U(Y) \cong K_9$; a contradiction. Thus, without loss of generality, we may assume that $|R_1| = 2 = |R_2|$.  If $|R_3| > 4$, then for $c_i \in R_3^*$, note that the set \[V = \{(0,0,0), (0,1, 0), (0, 0, c_1), (0, 0, c_2), (0, 0, c_3), (0, 0, c_4), (0, 1, c_1), (0, 1, c_2), (0, 1, c_3)\}\]  induces a subgraph which is isomorphic to $K_9$, a contradiction. Further, if $|R_3| = 4$, then observe that the graph $\Gamma_U(R)$ contains two blocks formed by the set of vertices \[\{(0,0,0), (0,1, 0), (0, 0, c_1), (0, 0, c_2), (0, 0, c_3), (0, 1, c_1), (0, 1, c_2), (0, 1, c_3), (1,0,0), (1,0, c_1), (1, 0, c_2), (1, 0, c_3)\};\] one is isomorphic to $K_8$ and the other is isomorphic to $K_5$. Therefore, by Lemma \ref{genusofblocks}, we get $\mathbb{g}(\Gamma_U(R)) > 2$, a contradiction. Consequently, $R$ is isomorphic to the ring $\mathbb{Z}_2 \times \mathbb{Z}_2 \times \mathbb{Z}_3$.

Now let $R \cong R_1 \times R_2$. If either  $|R_1| \geq 9$ or $|R_2| \geq 9$, then there exists an induced subgraph of $\Gamma_U(R)$ which is isomorphic to $K_9$; a contradiction. It follows that the cardinality of each $R_i$ is at most $8$. Now we characterize the rings $R \cong R_1 \times R_2$ such that $\Gamma_U(R)$ has genus $2$ through the following cases.

\noindent\textbf{Case-1:} $|R_2| = 8$. Suppose that $|R_1| = 8$. If both $R_1$ and  $R_2$ are fields, then  $v = 15$,  $e = 56$ and Lemma \ref{eulerformulagenus} gives $f = 39$; a contradiction. If $R_1$ is not a field, then the existence of $K_9$ as an induced subgraph of $\Gamma_U(R)$ gives $\mathbb{g}(\Gamma_U(R)) > 2$; a contradiction. Consequently, $|R_1| \neq 8$. 

\textbf{Subcase-1.1:} $R_2$ is not a field. By the table given in \cite{nowickitables}, the ring $R_2$ is isomorphic to one of the following rings: $\mathbb{Z}_8$, $\frac{\mathbb{Z}_2[x]}{(x^3)}, \frac{\mathbb{Z}_2[x, y]}{(x^2, xy,y^2)}, \frac{\mathbb{Z}_4[x]}{(2x, x^2)},$ $\frac{\mathbb{Z}_4[x]}{(2x, x^2-2)}$. First suppose that $R_2$ is isomorphic to either $\mathbb{Z}_8$ or $\frac{\mathbb{Z}_2[x]}{(x^3)}$. If $|R_1| \in \{3, 4, 5, 7\}$, then we can easily obtain $K_9$ as an induced subgraph of $\Gamma_U(R)$, a contradiction. If $|R_1| =2$, then note that $\Gamma_U(R)$ has $12$ vertices and $50$ edges. Consequently, by Lemma \ref{eulerformulagenus}, $\Gamma_U(R)$ contains $36$ faces, a contradiction to the Remark \ref{triangularface}.

Now suppose that $R_2$ is isomorphic to one of the following $3$ rings: $\frac{\mathbb{Z}_2[x, y]}{(x^2, xy,y^2)}, \frac{\mathbb{Z}_4[x]}{(2x, x^2)}, \frac{\mathbb{Z}_4[x]}{(2x, x^2-2)}$. If $|R_1| \in \{5, 7\}$, then $\Gamma_U(R)$ has an induced subgraph isomorphic to $K_9$; a contradiction. Next, assume that $R_1$ is not a field of cardinality $4$. Then note that the set $\{(0, 0), (0, b_1), \cdots, (0, b_7), (z, 0)\}$, where $z \in Z(R_1^*)$ and $b_1, b_2, \cdots, b_7 \in R_2^*$, induces a subgraph which is isomorphic to $K_9$; a contradiction. Let $R_1$ be a field of cardinality $4$. Then $\Gamma_U(R)$ has $20$ vertices and $97$ edges. Further by Lemma \ref{eulerformulagenus}, $\Gamma_U(R)$ has $75$ faces, a contradiction to Remark \ref{triangularface}. If $|R_1| = 3$, then $v= 16$, $e=64$ and by Lemma \ref{eulerformulagenus} we get $f = 46$; a contradiction. If $|R_1| = 2$ then note that $\Gamma_U(R)$ has $12$ vertices and $41$ edges. Let $u_1, u_2, u_3, u_4 \in U(R_2)$ and $z_1, z_2, z_3 \in Z(R_2^*)$. Note that \[V(\Gamma_U(R)) = \{(0,0), (0,u_1), (0,u_2), (0,u_3), (0, u_4), (0,z_1), (0,z_2), (0,z_3), (1,0), (1,z_1), (1,z_2), (1,z_3)\}.\] Let $f_i$ be the number of faces of
size $i$ in an embedding of $\Gamma_U(R)$. Since $\mathbb{g}(\Gamma_U(R)) = 2$, by Lemma \ref{eulerformulagenus}, we have $f = 27$. Note that $f_4 + 2f_5  = 2e-3f = 1$. It follows that any embedding of  $\Gamma_U(R)$ in $\mathbb{S}_2$ has $26$ triangular faces and one quadrilateral face. Now consider the set $X = \{(0,0), (0, u_1), (0,u_2), (0,u_3), (0, u_4), (0, z_1), (0,z_2), (0,z_3)\}$, we get $\Gamma_U(X) \cong K_8$. Consequently, any embedding of $\Gamma_U(X)$ has either one pentagonal, $17$ triangular faces or two quadrilateral, 16 triangular faces. Suppose $\Gamma_U(X)$ has one pentagonal, $17$ triangular faces in an embedding on $\mathbb{S}_2$. If we insert the set $Y = \{(1,0), (1,z_1), (1,z_2), (1,z_3)\}$ of vertices and their respective edges to the embedding of $\Gamma_U(X)$ then $\Gamma_U(Y)$ must be embedded in the pentagonal face. Consequently, any embedding of $\Gamma_U(R)$ on $\mathbb{S}_2$ leads to an edge crossing. Similarly, the insertion of $\Gamma_U(Y)$ in an embedding of $\Gamma_U(X)$, when it has two quadrilateral and 16 triangular faces, yields to an edge crossing. Therefore, $\mathbb{g}(\Gamma_U(R)) > 2$; a contradiction and so this subcase is not possible.

\textbf{Subcase-1.2:}  $R_2$ is a field.  If $|R_1| \in \{5, 7\}$, then by Lemma \ref{genusofblocks} and Remark \ref{joinandunionoftwo graphs}, we obtain $\mathbb{g}(\Gamma_U(R)) \neq 2$; a contradiction. If $|R_1| = 4$ such that $R_1$ is not a field, then there exists a $z \in Z(R_1^*)$. Consequently, for the set $Y' = \{(0,0), (0,1), \cdots, (0,7), (z, 0)\}$, we have $\Gamma_U(Y') \cong K_9$; a contradiction. Thus, $R$ is isomorphic to one of the following $3$ rings: $\mathbb{Z}_2 \times \mathbb{F}_8,$ $ \mathbb{Z}_3 \times \mathbb{F}_8,$ $ \mathbb{F}_4 \times \mathbb{F}_8$.
  
\noindent\textbf{Case-2:} $|R_2|\in \{5, 7\}$. If $R_1$ is a field such that $|R_1| \leq 4$, then by Theorem \ref{genus1result}, we have $\mathbb{g}(\Gamma_U(R)) = 1$. If $|R_1| = 4$ and $R_1$ is not a field, then for $z \in Z(R_1^*)$ and $A = \{(0, 0), (0, 1), (0, 2), (0, 3), (0, 4), (z, 0), (z, 1), (z, 2), (z, 3)\}$, note that $\Gamma_U(A) \cong K_9$,  Consequently, $R$ is isomorphic to one of the following $3$ rings: $\mathbb{Z}_7 \times \mathbb{Z}_7, \mathbb{Z}_5 \times \mathbb{Z}_7, \mathbb{Z}_5 \times \mathbb{Z}_5$.

\noindent\textbf{Case-3:} $|R_2| = 4$. Assume that $|R_1| = 4$. If both $R_1$ and $R_2$ are fields, then by Theorem \ref{Planarupperideal}, $\Gamma_U(R_1 \times R_2)$ is planar; a contradiction. We may now suppose that both $R_1$ and  $R_2$ are not fields. Then note that $\Gamma_U(R)$ has $12$ vertices and $50$ edges. Consequently, by Remark \ref{triangularface}, $\Gamma_U(R)$ has $36$ faces, a contradiction. Note that by Theorem \ref{Planarupperideal} and Theorem \ref{genus1result}, we get $|R_1| \not \in \{2, 3\}$. Thus, $|R_1| = 4$ and both $R_1$, $R_2$ cannot be fields. Consequently, either $R \cong \mathbb{F}_4 \times \mathbb{Z}_4$ or $R \cong \mathbb{F}_4 \times \frac{\mathbb{Z}_2[x]}{(x^2)}$.

\noindent\textbf{Case-4:} $|R_2| \leq 3$. If $|R_1| \leq 3$, then by Theorem \ref{Planarupperideal}, the graph $\Gamma_U(R)$ is planar. 

Conversely, suppose that $R \cong \mathbb{Z}_2 \times \mathbb{Z}_2 \times \mathbb{Z}_3$.  By Figure \ref{genustwo223}, $\mathbb{g}(\Gamma_U(R)) = 2$. If either $R \cong \mathbb{F}_4 \times \mathbb{Z}_4,$ or $R \cong \mathbb{F}_4 \times \frac{\mathbb{Z}_2[x]}{(x^2)}$, then  $\Gamma_U(R) \cong  K_2 \vee (K_2 \bigcup K_6)$. Consequently, by Proposition \ref{genus} and Lemma \ref{genusofblocks}, we obtain $\mathbb{g}(\Gamma_U(R)) = 2$. Further, if $R$ is isomorphic to one of the remaining $6$ rings, then by Lemma \ref{genusofblocks} and Remark \ref{joinandunionoftwo graphs}, we obtain $\mathbb{g}(\Gamma_U(R)) = 2$.  
\begin{figure}[h!]
\centering
\includegraphics[width=0.5 \textwidth]{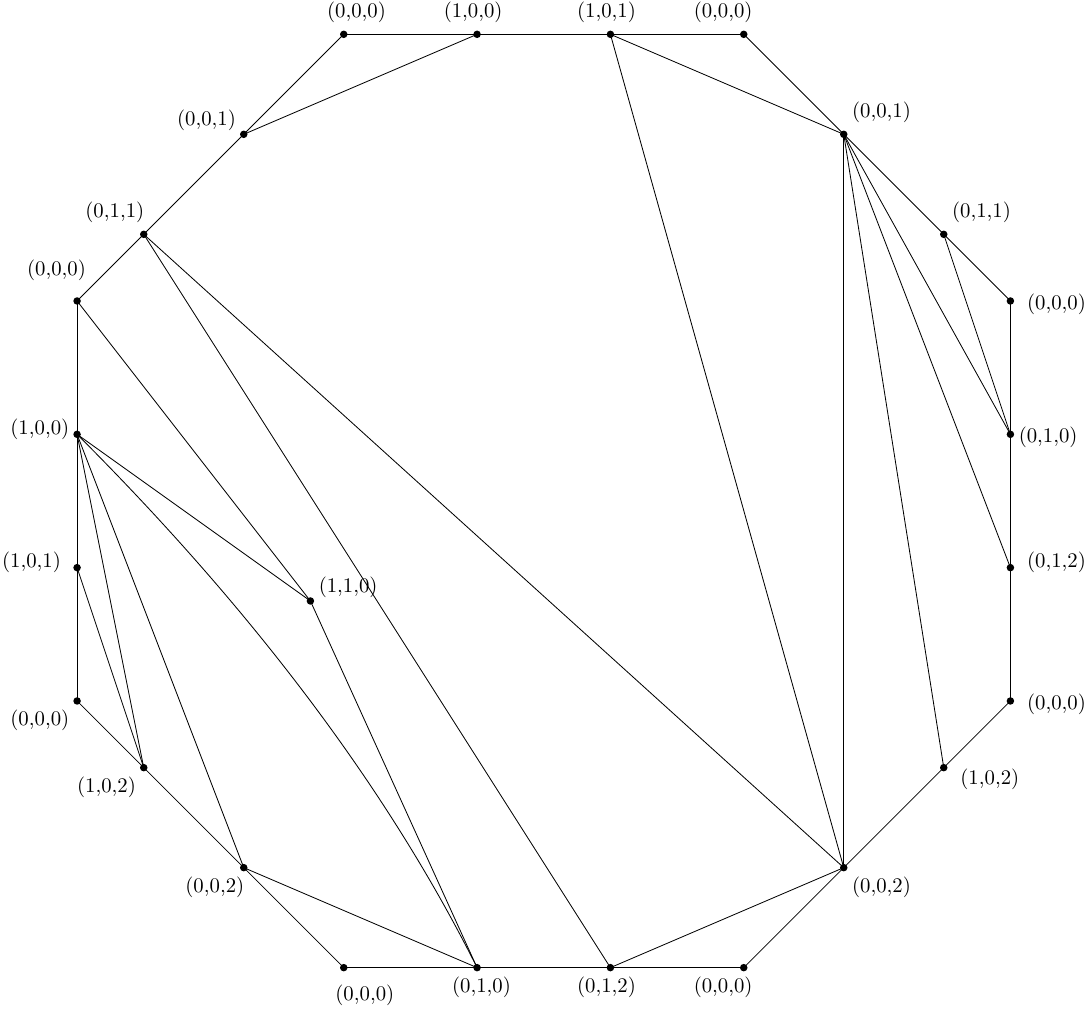}
			\caption{ Embedding of  $\Gamma_U(\mathbb{Z}_2 \times \mathbb{Z}_2 \times \mathbb{Z}_3)$  in $\mathbb{S}_2$}
			\label{genustwo223}
\end{figure}
\end{proof}
\subsection{Crosscap of $\Gamma_U(R)$}
In this subsection, we characterize all the non-local finite commutative rings such that $\Gamma_U(R)$ has crosscap $1$ or $2$. We begin with the following lemma.

\begin{lemma}\label{crosscaplemma1}
Let $R$ be a non-local finite commutative ring such that $R \cong R_1 \times R_2 \times \cdots \times R_n$, for $n \geq 4$. Then $cr(\Gamma_U(R) > 2$.
\end{lemma}
\begin{proof}
Let $n \geq 4$. Then note that the vertices $x_1 = (0, 0, \cdots, 0),$  $x_2 = (1, 0, \cdots, 0),$ $x_3 = (0, 1, 0, \cdots, 0),$\\
$x_4 = (0, 0, 1, 0, \cdots, 0),$  $x_5 =(1, 1, 0, 0, \cdots, 0),$ $x_6 = (1, 0, 1, 0, \cdots, 0),$ $x_7 = (0, 1, 1, 0, \cdots, 0),$ $x_8 = (1, 1, 1, 0, \cdots, 0)$ induces a subgraph of $\Gamma_U(R)$ which is isomorphic to $K_8$. Therefore, by Proposition \ref{crosscap}, we get $cr(\Gamma_U(R)) > 2$.
\end{proof}
In the following theorem, we precisely provide all the non-local finite commutative rings $R$ such that the graph $\Gamma_U(R)$ is of crosscap one.

\begin{theorem}\label{crosscap1result}
Let $R$ be a non-local finite commutative ring. Then the crosscap of $\Gamma_U(R)$ is $1$ if and only if $R$ is isomorphic to one of the following $5$ rings:
\begin{center}
    $\mathbb{Z}_2 \times \mathbb{Z}_5,$ $\mathbb{Z}_3 \times \mathbb{Z}_5,$ $\mathbb{F}_4 \times \mathbb{Z}_5,$  $\mathbb{Z}_3 \times \mathbb{Z}_4,$ $\mathbb{Z}_3 \times \frac{\mathbb{Z}_2[x]}{(x^2)}$.
\end{center}
\end{theorem}
\begin{proof}
Let $R$ be a non-local finite commutative ring. Then $R \cong R_1 \times R_2 \times \cdots \times R_n$, where each $R_i$ is a local ring and $n \geq 2$. Suppose that $cr(\Gamma_U(R)) = 1$. In view of  Lemma \ref{crosscaplemma1}, either $R \cong R_1 \times R_2 \times R_3$ or $R \cong R_1 \times R_2$. First assume that $R \cong R_1 \times R_2 \times R_3$. Suppose that $|R_i| \geq 3$ for all $i \in \{1, 2, 3\}$. Let $a_1, a_2 \in R_1^*$, $b_1, b_2 \in R_2^*$ and $c_1, c_2 \in R_3^*$. Then for the set $X = \{(0, 0, 0), (a_1, 0, 0), (0, b_1, 0), (a_2, 0, 0), (a_1, b_1, 0), (0, b_2, 0), (a_1, b_2, 0)\}$, we obtain $\Gamma_U(X) \cong K_7$; a contradiction. Without loss of generality, let $|R_2| \geq 3$ and $|R_3| \geq 3$. Then the set $ \{(0, 0, 0), (0, 0, c_1), (0, b_1, 0), (0, 0, c_2), (0, b_1, c_1), (0, b_2, 0), (0, b_2, c_1)\}$ induces $K_7$ as a subgraph of $\Gamma_U(R)$, which is not possible. Without loss of generality, we may assume that $|R_1| = 2 = |R_2|$ and $|R_3| \geq 3$. If $|R_3| > 3$, then the set $ \{(0, 0, 0), (0, 0, c_1), (0, 0, c_2), (0, 0, c_3), (0, 1, c_1), (0, 1, 0), (0, 1, c_2)\}$, where $c_1, c_2, c_3 \in R_3^*$, induces a subgraph isomorphic to $K_7$, which is a contradiction. If $|R_3| = 3$, then $\Gamma_U(R)$ has $10$ vertices and $31$ edges. By Lemma \ref{eulerformula}, we have  $f = 22$, a contradiction to Remark \ref{triangularface}. Therefore, $|R_i| = 2$ for each $i$ and so $R \cong \mathbb{Z}_2 \times \mathbb{Z}_2 \times \mathbb{Z}_2$. By Figure \ref{planardrawingof222}, the graph $\Gamma_U(R)$ is planar.
Thus, $R \cong R_1 \times R_2$. If either $|R_1| \geq 7$ or $|R_2| \geq 7$ then we can easily get $K_7$ as an induced subgraph of $\Gamma_U(R)$. By Proposition \ref{crosscap}, $cr(\Gamma_U(R)) \neq 1$; a contradiction. In view of Remark \ref{localringcardinality}, we classify $R$ through the following cases.

\noindent\textbf{Case-1:} $|R_2| = 5$. By Remark \ref{joinandunionoftwo graphs} and Lemma \ref{crosscapofblocks}, we have $|R_1| \neq 5$. It follows that $|R_1| \leq 4$. Further, if $R_1$ is not a field of cardinality $4$, then there exists a  $z \in Z(R_1^*)$. Note that the set $X' = \{(0, 0), (0, 1), \cdots, (0, 4), (z, 0), (z, 1)\}$ induces $K_7$ as a subgraph of $\Gamma_U(R)$; again a contradiction. Thus, $R$ is isomorphic to one of the following three rings: $\mathbb{Z}_2 \times \mathbb{Z}_5,$ $\mathbb{Z}_3 \times \mathbb{Z}_5,$ $\mathbb{F}_4 \times \mathbb{Z}_5$. 

\noindent\textbf{Case-2:}  $|R_2| = 4$.
Let $|R_1| = 4$. If both $R_1$ and $R_2$ are fields, then by Theorem \ref{Planarupperideal}, the graph $\Gamma_U(R)$ is planar. Without loss of generality, assume that $R_1$ is not a field.  Then for $z \in Z(R_1^*)$, note that the set $Y = \{(0, 0), (0, b_1), (0, b_2), (0, b_3), (z, 0), (z, b_1), (z, b_2)\}$ induces $K_7$ as a subgraph of $\Gamma_U(R)$; a contradiction. Therefore, $|R_1| \neq 4$. Further, suppose that  $|R_1| \in \{2, 3\}$ and $R_2$ is a field. Then by Theorem \ref{Planarupperideal}, $\Gamma_U(R)$ is planar. Next, if $|R_1| = 2$ and $R_2$ is not a field, then by Theorem \ref{Planarupperideal}, we have $cr(\Gamma_U(R)) = 0$; again a contradiction. Thus, either  $R \cong \mathbb{Z}_3 \times \mathbb{Z}_4$  or $R \cong \mathbb{Z}_3 \times \frac{\mathbb{Z}_2[x]}{(x^2)}$.

By Theorem \ref{Planarupperideal}, for each $i \in \{1, 2\}$, note that $|R_i| \leq 3$ is not possible. Conversely, if $R \cong \mathbb{F}_4 \times \mathbb{Z}_5$, then from Figure \ref{crosscap1of45}, $cr(\Gamma_U(R)) = 1$. For $R \cong \mathbb{Z}_2 \times \mathbb{Z}_5$, we get $\Gamma_U(\mathbb{Z}_2 \times \mathbb{Z}_5) \cong \Gamma_U(X)$, where $X = V(\Gamma_U(\mathbb{F}_4 \times \mathbb{Z}_5)) - \{(2,0), (3,0)\}$. It follows that $cr(\Gamma_U(\mathbb{Z}_2 \times \mathbb{Z}_5)) = 1$. Similarly, $cr(\Gamma_U(\mathbb{Z}_3 \times \mathbb{Z}_5)) = 1$ because $\Gamma_U(\mathbb{Z}_3 \times \mathbb{Z}_5) \cong \Gamma_U(Y)$ for $Y = V(\Gamma_U(\mathbb{F}_4 \times \mathbb{Z}_5)) - \{(3,0)\}$. Finally, if either $R \cong \mathbb{Z}_3 \times \mathbb{Z}_4$ or $R \cong \mathbb{Z}_3 \times \frac{\mathbb{Z}_2[x]}{(x^2)}$, then by Figure \ref{crosscap1of34}, we obtain $cr(\Gamma_U(R)) = 1$.
\begin{figure}[h!]
\centering
\includegraphics[width=0.4 \textwidth]{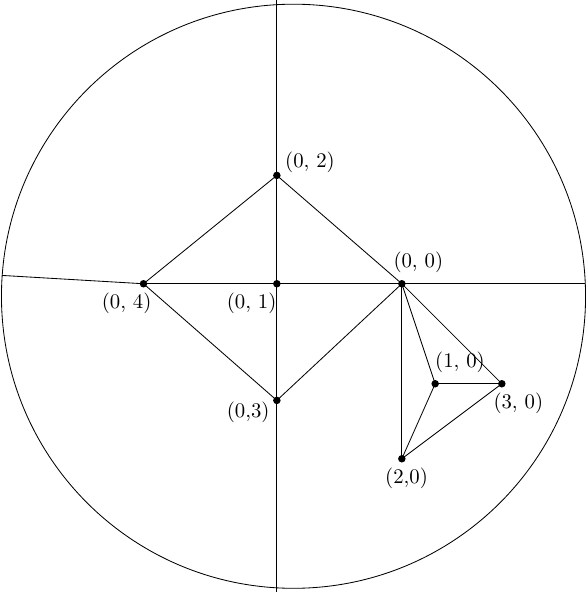}
			\caption{ Embedding of $\Gamma_U(\mathbb{F}_4 \times \mathbb{Z}_5)$  in $\mathbb{N}_1$}
			\label{crosscap1of45}
\end{figure}
\begin{figure}[h!]
\centering
\includegraphics[width=0.8 \textwidth]{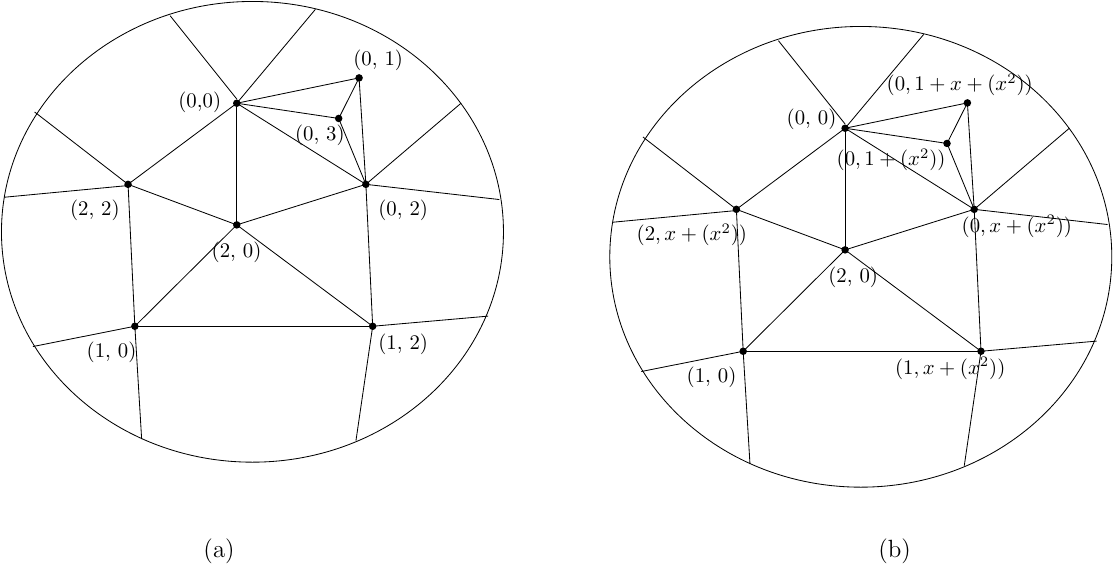}
			\caption{ Embedding of (a) $\Gamma_U(\mathbb{Z}_3 \times \mathbb{Z}_4)$ and (b) $\Gamma_U(\mathbb{Z}_3 \times \frac{\mathbb{Z}_2[x]}{(x^2)})$  in $\mathbb{N}_1$}
			\label{crosscap1of34}
\end{figure}
\end{proof}
The following theorem illustrates that the ring $\mathbb{Z}_5 \times \mathbb{Z}_5$ is the only ring whose upper ideal relation graph has crosscap two.
\begin{theorem}
Let $R$ be a non-local finite commutative ring. Then $cr(\Gamma_U(R)) = 2$ if and only if $R$  is isomorphic to $\mathbb{Z}_5 \times \mathbb{Z}_5$.
\end{theorem}
\begin{proof}
Let $cr(\Gamma_U(R)) = 2$ and $R$ be a non-local finite commutative ring. Then $R \cong R_1 \times R_2 \times \cdots \times R_n$, where $n \geq 2$ and each $R_i$ is a local ring.
 By Lemma \ref{crosscaplemma1}, we get $n \leq 3$. First suppose that $n = 3$, that is $R \cong R_1 \times R_2 \times R_3$. Let $|R_i| = 2$ for all $i \in \{1, 2, 3\}$. Then by Figure \ref{planardrawingof222}, we have $cr(\Gamma_U(R)) = 0$, a contradiction. Without loss of generality, we may now assume that $|R_3| \geq 3$. If $|R_3| > 3$, then by using the set 
 $X = \{(0, 0, 0), (0, 1, 0), (0, 0, c_1), (0, 0, c_2), (0, 0, c_3), (0, 1, c_1), (0, 1, c_2)\}$, where $c_1, c_2, c_3 \in R_3^{*}$,  we get $\Gamma_U(X) \cong K_7$; a contradiction. Now suppose that $|R_3| = 3$. If $|R_i| = 3$ for some $i \in \{1,2\}$, then observe that there exists $K_7$ as an induced subgraph of $\Gamma_U(R)$. It follows that $|R_1| = |R_2| = 2$. Since $|R_3| = 3$, note that $v = 10$, $e = 31$ and then by Lemma \ref{eulerformula}, we get $f = 21$; a contradiction to the Remark \ref{triangularface}.
 Thus, $R \cong R_1 \times R_2$. If either $|R_1| \geq 7$ or $|R_2| \geq 7$, then one can get $K_7$ as an induced subgraph of $\Gamma_U(R)$ and so $cr(\Gamma_U(R)) \neq 2$; a contradiction. Therefore, by Remark \ref{localringcardinality},  $|R_i| \leq 5$ for each $i = 1, 2$. Now we have the following cases:

\noindent\textbf{Case-1:} $|R_2| = 5$. Let $R_1$ be a field such that $|R_1| \in \{2, 3, 4\}$. Then by Theorem \ref{crosscap1result}, we get $cr(\Gamma_U(R)) = 1$. If $R_1$ is not a field of cardinality $4$, then by the Case-1 of Theorem \ref{crosscap1result}, we get a contradiction. Therefore, $R \cong \mathbb{Z}_5 \times \mathbb{Z}_5$.

\noindent\textbf{Case-2:} $|R_2| = 4$. Suppose that $|R_1| \in \{2, 3, 4\}$. If both $R_1$ and $R_2$ are fields, then by Theorem \ref{Planarupperideal}, the graph $\Gamma_U(R)$ is planar. Now suppose that $R_2$ is not a field. If $|R_1| \in \{2, 3 \}$, then by Theorem \ref{Planarupperideal} and Theorem \ref{crosscap1result}, we have $cr(\Gamma_U(R)) \neq 2$; a contradiction. If $|R_1| = 4$, then the set $Y = \{(0, 0), (b_1, 0), (b_2, 0), (b_3, 0), (0, z), \cdots, (b_2, z)\}$, where $z \in Z(R_2^*)$ and $b_1, b_2, b_3 \in R_1^*$, induces $K_7$ as a subgraph of $\Gamma_U(R)$. Thus, this case is not possible.

In view of the above cases and by Theorem \ref{Planarupperideal}, note that the case $|R_2| \leq 3$ is not possible. Conversely if $R \cong \mathbb{Z}_5 \times \mathbb{Z}_5$, then by Remark \ref{joinandunionoftwo graphs} and Lemma \ref{crosscapofblocks}, we obtain $cr(\Gamma_U(R)) = 2$.
\end{proof}
\section{Conclusion and Future Research}
In this manuscript, we precisely determined all the non-local finite commutative rings $R$ whose upper ideal relation graph $\Gamma_U(R)$ has genus or crosscap at most two. Moreover, all the non-local finite commutative rings $R$ have been ascertained such that $\Gamma_U(R)$ is a split graph, threshold graph and cograph, respectively. The work of this manuscript is limited to non-local finite commutative rings. To reveal some more interconnections between the ring $R$ and its upper ideal relation graph $\Gamma_U(R)$, the following problems can be explored in future.
\begin{itemize}
    \item Characterization of local rings $R$ such that $\Gamma_U(R)$ is of genus (or crosscap) at most two.

\item Classification of rings $R$ such that $\Gamma_U(R)$ is a line graph. The results obtained in Section $3$ of the manuscript will be useful for such study.

\item To determine the automorphism group of the graph $\Gamma_U(R)$, when $R$ is the ring of matrices.

\item The perfectness of the upper ideal relation graph $\Gamma_U(R)$ of the non-local finite commutative ring has been investigated in \cite{baloda2023study}. The perfectness of the upper ideal relation graph $\Gamma_U(R)$, where $R$ is a local ring, has not been investigated yet.
\end{itemize}

\vspace{.3cm}
\textbf{Acknowledgements}: We would like to thank the referees for valuable suggestions which improved the presentation of this paper.\\

\section*{Declarations}

\textbf{Funding}: The first author gratefully acknowledge for providing financial support to CSIR  (09/719(0093)/2019-EMR-I) government of India. The second author wishes to acknowledge the support of MATRICS Grant  (MTR/2018/000779) funded by SERB, India.

\vspace{.3cm}
\textbf{Conflicts of interest/Competing interests}: There is no conflict of interest regarding the publication of this paper. 

\vspace{.3cm}
\textbf{Availability of data and material (data transparency)}: Not applicable.

\vspace{.3cm}
\textbf{Code availability (software application or custom code)}: Not applicable.

\end{document}